\pdfoutput=1

\documentclass[%
aip,cha,
 amsmath,amssymb,
 reprint,%
]{revtex4-1}

\newcommand{\rev}[1]{{#1}}

\newcommand{\revv}[1]{{#1}}

\usepackage{graphicx}
\usepackage{dcolumn}
\usepackage{bm}

\usepackage[utf8]{inputenc}
\usepackage[T1]{fontenc}
\usepackage{mathptmx}

\usepackage{mathrsfs}   

\usepackage[caption=false]{subfig}

\usepackage{color}  		
\usepackage{tabularx} 
\usepackage{wasysym} 

\usepackage{amsthm}
\usepackage[capitalize,nameinlink,noabbrev]{cleveref}

\newtheoremstyle{theoremstyle1}
{10pt}
{17pt}
{}
{}
{\itshape}
{}
{\newline}
{\thmname{#1}\thmnumber{ #2}\thmnote{ (#3)}}

\theoremstyle{theoremstyle1}
\newtheorem{definition}{Definition}[section]
\newtheorem{theorem}{Theorem}[section]
\newtheorem{remark}{Remark}[section]

\newtheorem{corollary}{Corollary}[section]
\newtheorem{lemma}{Lemma}[section]
\newtheorem{proposition}{Proposition}[section]

\definecolor{gr}{rgb}{0.5,0.5,0.5} 
\newcommand{\inh}{w}
\newcommand{\ks}{k_s}
\newcommand{\imag}{{\textrm{i}}}
\newcommand{\newpar}{\\ \par}
\newcommand{\xI}{x_1}
\newcommand{\xII}{x_2}

\begin{document}

\preprint{AIP/123-QED}

\title[Time Delay in the Swing Equation: A Variety of Bifurcations]{
Time Delay in the Swing Equation: 
A Variety of Bifurcations}

\author{T. H. Scholl}
 \email{tessina.scholl@kit.edu}
\author{L. Gröll}%
 \email{lutz.groell@kit.edu}
\author{V. Hagenmeyer}
 \email{veit.hagenmeyer@kit.edu}
\affiliation{ 
Institute for Automation and Applied Informatics, Karlsruhe Institute of Technology, \\
76344 Eggenstein-Leopoldshafen, Germany
}

\date{2019}

\begin{abstract}
The present paper addresses the swing equation with additional delayed damping as an example for pendulum-like systems. In this context, it is proved that recurring sub- and supercritical Hopf bifurcations occur if time delay is increased. To this end, a general formula for the first Lyapunov coefficient in second order systems with additional delayed damping and delay-free nonlinearity is given. In so far the paper extends results about stability switching of equilibria in linear time delay systems from Cooke and \revv{Grossman}\rev{.}  In addition to the analytical results, periodic solutions are numerically dealt with. The numerical results demonstrate how a variety of qualitative behaviors is generated in the simple swing equation by only introducing time delay in a damping term.
\end{abstract}

\maketitle

\begin{quotation}
The swing equation is, e.g., at the core of any power system model. By delayed frequency control it becomes a pendulum equation with delayed damping - a system which is shown to exhibit highly complex dynamics. Repeating Hopf bifurcations lead again and again to the emergence of limit cycles as delay is increased. These limit cycles undergo further bifurcations including period doubling cascades and the birth of invariant tori. 
\end{quotation}


\section{\label{sec:}Introduction}
Second order ordinary differential equations with periodic nonlinearity
describe various non-related systems such as \rev{synchronous generators or  single-machine-infinite-bus (SMIB) equivalents of power systems, Josephson junction circuits, phase-locked loops or mechanical pendulum mechanisms\cite{Khalil.2002}}. \rev{Consider the scalar equation
\begin{align}
\ddot y+a  \dot y+\ks \sin(y)&=u
\end{align}
with an external input $u$.
Either in order to increase damping or to control the frequency (if $y$ is expressed w.r.t.\ a rotating reference frame), $u$ 
might involve a derivative-dependent feedback $u=\inh-\tilde a \dot y$. 
 However, due to measurement, communication, data processing, or response behavior of actuators,  this controller action is more or less delayed. Consequently, damping with delay $\tau>0$ arises in }
\begin{align}
\ddot y(t)+a  \dot y(t)+\tilde a \dot y(t-\tau) + \ks \sin(y(t))=\inh, 
\label{eq:main}
\end{align}
$a,\tilde a, k_s, \inh> 0$.
Schaefer et al.\cite{Schafer.2015} \rev{show that these dynamics might become of high technical relevance in a future power system, where time delay is caused by frequency-dependent reactions of the demand side. In this smart grid, $y(t)\in \mathbb R$ represents an  aggregated} relative  generator angle  w.r.t.\ a rotating reference\cite{Schafer.2015}.
However, the thereby generated system dynamics are only described in a first attempt in Schaefer et al.\cite{Schafer.2015}\rev{. Hence,} a thorough analysis of these dynamics is an open question. 
\\
\\
Therefore, the present paper examines how the delay changes the overall system dynamics and gives the respective results (throughout the paper, the term increase denotes not a time-varying sense but addresses another autonomous system with changed parametrization).
Dynamics of the delay-free swing equation, or equivalently the driven pendulum equation with constant torque, is well known\cite{Schaeffer.2016,Andronov.1966,Leonov.1996,Manik.2014,Levi.1978} (Figure \ref{fig:PhasePortrait}).
All additional complex phenomena which are made possible by time delay must successively emerge with increasing delay. 
Hence, we start with the delay-free system and tackle the problem by a bifurcation analysis with time delay as the bifurcation parameter. Foundation is laid by bifurcation theory of retarded functional differential equations (RFDEs)\cite{Diekmann.1995,Hale.1993}. 
\begin{figure}
\includegraphics[]{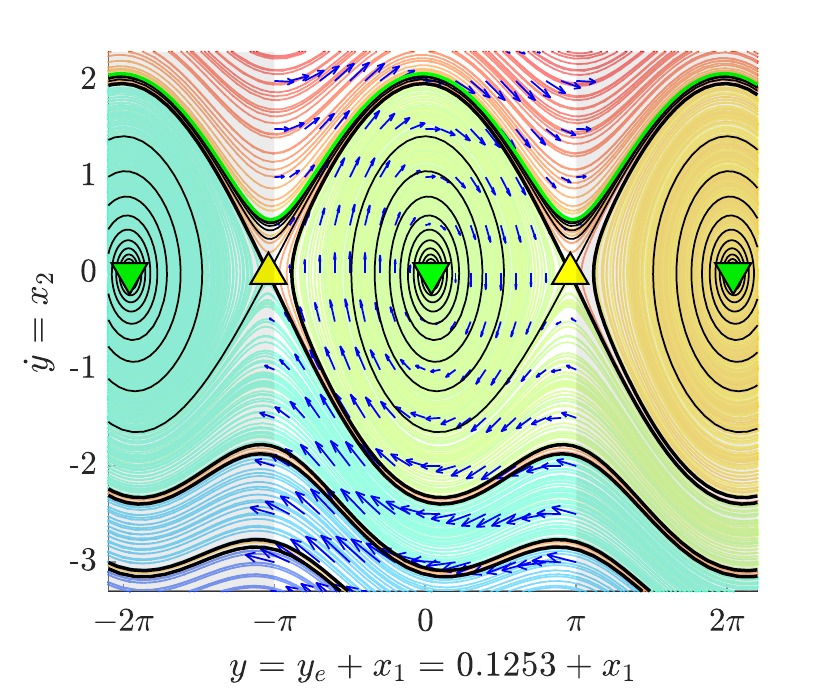}
\caption{\label{fig:PhasePortrait} Phase portrait of the delay-free dynamics (parameterization (\ref{eq:param}), $\tau=0$), i.e. damping coefficient $a+\tilde a$. }
\end{figure} 
\newpar
The linearization of (\ref{eq:main}) belongs to the general class of linear second order systems with delayed damping\rev{\cite{Cooke.1982,Gilliam.2002,Sherman.1947,Ansoff.1948,Pinney.1958,Minorsky.1942,Minorsky.1948,Erneux.2009}}.  
For these linear systems, Cooke and \revv{Grossman} \cite{Cooke.1982} have already proved alternate switching between stability and instability with increasing delay. Such switching points in the linearization are candidates for Hopf bifurcations in the nonlinear system (\ref{eq:main}). The latter is a well known fact from ordinary differential equations (ODEs) and indeed also applies to the  RFDE (\ref{eq:main}).  To the authors' best knowledge, systems with delayed damping \rev{and delay-free nonlinearity $h(y(t))$} have not yet been subject to bifurcation analysis in literature. Besides of delayed damping $\tilde a \dot y(t-\tau)$, $\tilde a \in \mathbb R_{>0}$ Cooke and \revv{Grossman} \cite{Cooke.1982} have proved similar stability switching behavior for 
linearized systems with delayed restoring force $\tilde c y(t-\tau)$, $\tilde c\in \mathbb R_{>0}$. 
   For this related system class, results are already given by Campbell at al.\cite{Campbell.1995c}, revealing not only limit cycles but also invariant tori. \rev{Similar observations are also valid for  Minorsky's equation\cite{Erneux.2009}, which in contrast to (\ref{eq:main}) exhibits a derivative-dependent and overall delayed nonlinear term which does not depend on $y(t)$. }
The occurrence of oscillatory problems due to limit cycles\cite{Mithulananthan.2000}, invariant tori\cite{Ji.1995,Tan.1995} and even period doubling cascades  which lead to chaos\cite{Tan.1995} are also relevant in the context of power systems. However, their observation requires higher order models. The present paper aims to show that only by introducing time delay, the simple swing equation alone is able to reveal  a rich variety of such phenomena.
\newpar
\rev{Moreover, the presented result on the first Lyapunov coefficient (Theorem \ref{th:LyapCoeffRes})  is valid for more general systems
\begin{align}
\ddot y(t)+a \dot y(t) + \tilde a \dot y(t-\tau) +  h(y(t))=0.
\label{eq:main_general}
\end{align}
This structure arises whenever a feedback law $u$ for a plant
\begin{align}
\ddot y + a \dot y +  cy+h_0(y)=u
\end{align}
uses a delayed measurement of the derivative.
For instance, a PD-like controller with static feedforward $w$  
\begin{align}
u(t)=w- k_P y(t)-k_D \dot y(t-\tau)
\end{align}
(state feedback realization with delay in $x_2=\dot y$)
results in (\ref{eq:main_general}) with $\tilde a =k_D$ and $h(y)=c y+h_0(y)-w+k_Py$.  }
\newpar
The paper is organized as follows. Subsequently, the used notation and needed preliminaries are given. Section \ref{sec:BasicAnalysis}
provides basic insights to the system dynamics by an analysis of the linearized system. Section \ref{sec:BifurcationAnalysis} addresses analytical as well as  numerical results on equilibrium bifurcations and a numerical treatment of limit cycle bifurcations. 

\subsection{Notation and Preliminaries}
We are concerned with an {\itshape autonomous difference differential  equation}\cite{Bellman.1963,Hahn.1967} $\dot x(t)=f(x(t),x(t-\tau))$, with $x(t)\in \mathbb R^n$ and a single discrete delay $\tau>0$. Consequently, not only the instantaneous value $x(t)\in \mathbb R^n$ which is an element of the $n$-dimensional Euclidean space  is decisive for the dynamics. Instead the delay-wide last solution segment $x(\tilde t)$, $\tilde t\in[t-\tau,t]$ has to be considered. Such delay-spanning solution segments are commonly denoted as $x_t(\theta) \stackrel{\textrm{def}}{=}x(t+\theta)$, $\theta\in[-\tau,0]$. 
The more general system class are {\itshape retarded functional differential equations (RFDE)} 
$\dot x(t)= F(x_t)$ 
, $F\colon X\to \mathbb R^n$. A Cauchy problem requires an initial state $x_0=\phi\in X$, i.e.\ an initial function segment. Obviously, the latter has to stem from the infinite dimensional state space which is usually a predefined Banach space $X$. A very common state space for time delay systems is the space of continuous functions $C([-\tau,0],\mathbb R^n)$ endowed with the uniform norm $\|\phi \|_C
=\max_{\theta\in [-\tau,0]}(\|\phi(\theta)\|_2)$. A zero equilibrium of an autonomous RFDE is 
locally asymptotically stable  (LAS) in $X$, if 
it is {\itshape stable} in $X$, i.e.\ 
$
\forall \varepsilon>0, \;\exists \delta(\varepsilon)>0: \; \|\phi\|_X \revv{<}\delta  \; \Rightarrow \|x_t\|_X \revv{<} \varepsilon, \; \forall t\geq 0
$
and {\itshape locally attractive} in $X$, i.e.\ 
$
\exists \delta_a>0: \quad \|\phi\|_X\revv{<} \delta_a  \; \Rightarrow \lim_{t\to \infty} \|x_t\|_X=0.
$
The {\itshape domain of attraction} $\mathcal D_{X}$ of an asymptotically stable zero equilibrium 
 is the set of all initial functions which lead to a zero-convergent solution, i.e.\ 
$
\mathcal D_{X}=\{\phi\in X: x_t \text{ exists on } t\in[0,\infty) \text{ and }\lim_{t\to \infty}\|x_t\|_X = 0  \}. 
$
\newpar
The {\itshape Principle of Linearized Stability} (Diekmann\cite{Diekmann.1995}, Chapter VII, Theorem 6.8) allows to conclude stability properties of equilibria based on a respective linearization about them. 
For well-posed autonomous linear RFDEs in a suited Banach space, the family $\{\mathscr T(t)\}_{t\geq 0}$ of solution operators $\mathscr T(t): X \to X$, $x_0 \mapsto x_t$ forms a {\itshape $C_0$-semigroup}. The spectrum of its {\itshape infinitesimal generator} $\mathscr A$ is a pure point spectrum $\sigma(\mathscr A)$. These eigenvalues $\lambda \in \sigma(\mathscr A)$ decide about stability of the linear system and are equal to the {\itshape characteristic roots}. Such characteristic roots can be derived analogously to the ODE case, i.e.\ either based on Laplace transformation of the linear RFDE, or on the ansatz $x(t)=v\textrm e^{\lambda t}$, $v\in \mathbb R^n$. 
Concerning limit cycle stability, a linearization about the periodic solution results in a linear time-periodic RFDE with family of solution operators $\{\mathscr T(t,t_0)\}_{t\geq t_0}$ as an analogue to the state transition matrix in ODEs.
It should be noted that a complete Floquet theory does not exist for RFDEs\cite{Hale.1993},
but space decompositions w.r.t generalized eigenspaces are very helpful\cite{Hale.1993,Diekmann.1995}.
Similarly to ODEs, the {\itshape Floquet multipliers}, i.e.\ eigenvalues of the {\itshape monodromy operator} $\mathscr T(T,0)$, are decisive for stability of linear time-periodic RFDEs with period $T$. Thereby, instability or asymptotic stability in the sense of Poincaré can be concluded for the limit cycle of the  nonlinear system. 
\section{\label{sec:BasicAnalysis}Basic Analysis}
In a first step, we will introduce a state space description of (\ref{eq:main}). 
W.l.o.g. we assume $\ks:=1$ in (\ref{eq:main}), since any other swing equation with arbitrary $\hat k_s>0$
\begin{align}
\hat y
''
(\hat t)
+\hat a  
\hat y
'
(\hat t)
+\hat{\tilde a} 
\hat y
'
(\hat t-\hat \tau) +\hat \ks \sin(\hat y(\hat t))=\hat \inh , 
\label{eq:dimless}
\end{align}
$(\,\cdot\,)'=\textrm d/(\textrm d {\hat t}) $
can be transformed by $\hat t =  t /{\sqrt{\hat \ks}} $ and $\hat y( t / {\sqrt{\hat \ks}} )=y(t)$ to (\ref{eq:main}) with dimensionless parameters 
\begin{align}
k_s=1, 
\hspace{0.9em} a=\frac{\hat a}{\sqrt{\hat \ks}}, 
\hspace{0.9em} \tilde a=\frac{\hat {\tilde a}}{\sqrt{\hat \ks}}, 
\hspace{0.9em} \inh=\frac{\hat \inh }{ \hat \ks}, 
\hspace{0.9em} \tau=\hat \tau {\sqrt{\hat \ks}}.
\end{align}
\subsection{State Space Description}
\rev{First, we introduce a relative coordinate $x_1=y-y_e$ with respect to some reference value $y_e$. Hence, if $y_e$ is chosen to be an equilibrium of interest, the common assumption, that this equilibrium lies in the origin, is already met in $x$-coordinates.}
For $x=[y-y_e,\dot y]^\top$, $y_e\in \mathbb R$, a state space representation of (\ref{eq:main}) with $k_s=1$ reads
\begin{alignat}{5}
\dot x(t)=& 
\begin{bmatrix} 0 & 1  \\ 0 & -a\end{bmatrix}
x(t) 
+
\begin{bmatrix} 0 & 0  \\ 0 & -\tilde a \end{bmatrix} 
x(t-\tau)
+&\begin{bmatrix} 0 \\ \inh-\sin(x_1(t)+y_e)\end{bmatrix}&
\nonumber\\
x(\theta) &=\phi(\theta), & \theta \in[-\tau,0] & 
\label{eq:delayedSwingEq}
\end{alignat}
$x(t)\in \mathbb R^2$, $t\in \mathbb R_{\geq (-\tau)}$, $a,\tilde a,\inh \in \mathbb R_{>0}$, $\tau \in \mathbb R_{\geq 0}$.
Thereby, $\phi\in C([-\tau,0],\mathbb R^2)$ describes the initial function.
\\
\par
We are interested in $\inh\in(0,1]$, since for $|\inh|> 1$ no equilibrium point exists. 
\rev{By assigning the principal value
\begin{align}
y_e:=\arcsin(\inh)
\label{eq:ye}
\end{align}
as reference, the equilibrium $y(t)\equiv \arcsin(\inh)$ in (\ref{eq:main}) is shifted to $x(t)\equiv [0,0]^\top$ in (\ref{eq:delayedSwingEq})}. Besides of $y(t)\equiv y_e+2k\pi$, $k\in \mathbb Z$ there are also equilibrium points in $y(t)\equiv (\pi-y_e) +2k\pi$, $k\in \mathbb Z$ (Figure \ref{fig:PhasePortrait}) - unless $\inh=1$ where both sets coincide. Consequently, in the coordinates of (\ref{eq:delayedSwingEq}), equilibria are $x(t)\equiv x_e$, $x_e\in \mathcal E_x$ with
\begin{align}
\mathcal E_x =\left\{\begin{bmatrix}2k\pi\\ 0\end{bmatrix}, k\in \mathbb Z\right\}
\cup  \left\{\begin{bmatrix}\pi-2y_e\\0\end{bmatrix}+\begin{bmatrix}2k\pi\\0\end{bmatrix}, k\in \mathbb Z
\right\},
\end{align}
if $\inh\in(0,1]$. 
By $X=C([-\tau,0],\mathbb R^2)$ the codomain of $\phi, x_t,x$ has been chosen to be the plane such that $x(t)\in \mathbb R^2$. In the plane, equilibria which differ in the angle value $x_1$ by $2\pi$ are considered as distinct equilibria. Hence, a slip over before settling down to $[2\pi,0]^\top$ will not contribute to the domain of attraction of $x_e=[0,0]^\top$, albeit the states are physically indistinguishable afterwards. 
This distinction would not be true on the cylinder $\mathbb S^1\times \mathbb R \ni [x_1(t),x_2(t)]^\top$, which is an alternative and quite natural choice due to periodicity in the above system. 
\newpar
Unless otherwise stated, we align the parametrization with Schaefer al al.\cite{Schafer.2015}. According to (\ref{eq:dimless}), the parameter set 
$\hat a=0.1 \,{\textrm s}^{-1}$, 
$\hat{\tilde a}= 0.25 \,{\textrm s}^{-1}$, 
$\hat \inh =2 \,{\textrm s^{-2}}$, 
$\hat k_s = 16  \,{\textrm s^{-2}}$ 
becomes 
\begin{align}
k_s=1, \quad a=0.025, \quad \tilde a=0.0625, \quad \inh=0.125.
\label{eq:param}
\end{align}
We are aware that considerations of small time scales serve as a legitimization for the strong simplification (\ref{eq:main}) as a power system model\cite{Grainger.1994,Bergen.2000,Anderson.1982,Kundur.1994}. 
However, in accordance with Schaefer at al.\cite{Schafer.2015} we will use (\ref{eq:main}) also on large time scales in order to create awareness for qualitative effects of a large delay term. A physical time delay of $\hat \tau=10\, \textrm s$ (Schaefer et al.\cite{Schafer.2015} 2015, Figure 5) results in $\tau=40$.   
\subsection{Linearization Based Stability Analysis}
Linearizing the delayed swing equation gives a linear second order differential equation with delayed damping. For such systems a stability analysis can already be found in literature\cite{Cooke.1982}.
The following lemma summarizes results about asymptotic stability and instability as far as the equilibrium is hyperbolic.
\begin{lemma}[Stability Switching \cite{Cooke.1982}]
\label{lem:stabSwitch}
Consider the linear RFDE
\begin{align}
\dot x(t)=& 
\begin{bmatrix} 0 & 1  \\ -c & -a\end{bmatrix}
x(t) 
+
\begin{bmatrix} 0 & 0  \\  0& -\tilde a \end{bmatrix} 
x(t-\tau)
\label{eq:linerRFDE_delayedDamping}
\end{align}
for fixed $a,\tilde a \in \mathbb R_{>0} , c \in \mathbb R$ 
and with a variable delay parameter $\tau>0$.\\
{\itshape (i)} If $c>0$, the system is asymptotically stable for $\tau=0$. 
\\
{\itshape (i-a) }If $\tilde a < a$, the system remains asymptotically stable for $\tau\geq0$. 
\\
{\itshape  (i-b) }If $a<\tilde a$, there are two sequences of delay values $(\tau_{1n})_{n\in\mathbb N}$ and $(\tau_{2n})_{n\in\mathbb N}$
\begin{subequations}
\begin{align}
\tau_{1n}=&\frac 1 {\omega_{1}} \left(\arccos \left(-\frac a {\tilde a}\right)+2\pi n \right) 
\label{eq:tau1n}
\\
\tau_{2n}=&\frac 1 {\omega_{2}} \left(-\arccos \left(-\frac a {\tilde a}\right)+2\pi (n+1) \right)
\label{eq:tau2n}
\end{align}
\label{eq:tau12n}
\end{subequations}
with corresponding 
\begin{align}
\omega_{1,2}= \pm \frac 1 2 {\sqrt{\tilde a^2- {a^2} }} 
+\sqrt{c+\frac 1 4 {(\tilde a^2- {a^2})} }
\label{eq:omega12}
\end{align}
such that the system switches to instability at $\tau_{1n}$ and back to asymptotic stability at $\tau_{2n}$ with increasing delay, as long as $\tau_{2n}<\tau_{1(n+1)}$.
If $\tau\in(\tau_{jn})_{n\in\mathbb N}$, $j\in\{1,2\}$ the equilibrium is non-hyperbolic 
and a simple complex conjugate eigenvalue pair $\{\imag \omega_j, -\imag \omega_j\}$ occurs.   \\
{\itshape (ii) } If $c=0$, the equilibrium is non-hyperbolic.
\\
{\itshape (iii) } If $c<0$, the system is unstable for $\tau=0$ and remains unstable $\forall \tau\geq0$.
\end{lemma}
The reader is referred to Cooke and \revv{Grossman} \cite{Cooke.1982} for a detailed proof.
We will shortly motivate equations (\ref{eq:tau12n}), (\ref{eq:omega12}) since we will later need some intermediate results.
\begin{proof}[Sketch of Proof for \itshape (i-b)]
Eigenvalues are roots of the characteristic quasipolynomial $\det(\Delta(\lambda))$ with
\begin{align}
\label{eq:Delta}
\Delta:\mathbb C \to \mathbb C^{2\times 2}, \quad \lambda \mapsto  \Delta(\lambda):=\begin{bmatrix} \lambda & -1  \\ c & \lambda+a+\tilde a \textrm{e}^{-\lambda \tau}  \end{bmatrix}. 
\end{align}
By continuity of the spectrum w.r.t.\ parameter changes (Michiels and Niculescu\cite{Michiels.2014}, Theorem 1.15) in RFDEs, a stability switch necessarily goes along with an eigenvalue $\lambda_i$ crossing the imaginary axis. 
Since 
\begin{align}
\lambda_i \neq 0, \quad \text{ if }c>0
\label{eq:lambda_neq_0}
\end{align}
this can only be accompanied by a complex conjugated eigenvalue pair on the imaginary axis $\{\lambda_i,\overline \lambda_i\}=\{\imag\omega,-\imag\omega\}$, $\omega>0$. Hence,  if the delay value $\tau$ should be a candidate point for stability switching, the characteristic equation $\det (\Delta(\lambda))=0$ must be fulfilled for $\lambda=\imag\omega$, $\omega > 0$. Thereby, $\det (\Delta(\imag\omega))=0$ is equivalent to
\begin{align}
\imag\omega ( a+\tilde a \textrm e^{-\imag\omega \tau} )&=\omega^2-c .
\label{eq:iwaexp}
\end{align}
Comparing imaginary parts of (\ref{eq:iwaexp})
\begin{align}
a+\tilde a \cos(\omega\tau)=0 
\label{eq:compImag}
\end{align}
leads to sequences $(\tau_{1n})_{n\in \mathbb N}$, $(\tau_{2n})_{n\in \mathbb N}$ with
\begin{align}
\omega_{1}\tau_{1n}=&\arccos \left(-\tfrac a {\tilde a}\right)+2\pi n
\label{eq:omega_tau}
\\
&\Rightarrow \quad
\sin(\omega_{1}\tau_{1n})= \sqrt{1-\tfrac {a^2} {\tilde a^2}} , \label{eq:sin1}
\\
\omega_{2}\tau_{2n}=&-\arccos \left(-\tfrac a {\tilde a}\right)+2\pi (n+1)
\\
&
\Rightarrow \quad 
\sin(\omega_{2}\tau_{2n})= -\sqrt{1-\tfrac {a^2} {\tilde a^2}}
,\label{eq:sin2}
\end{align}
while comparing real parts of (\ref{eq:iwaexp}) and using (\ref{eq:sin1}), (\ref{eq:sin2}) gives
\begin{align}
\omega^2 \mp  \tilde a \sqrt{1-\tfrac {a^2} {\tilde a^2}} \omega -c&=0 . 
\label{eq:compRe}
\end{align}
These two equations for  $\omega_1$ ($-$-case) and $\omega_2$ ($+$-case) with the restriction to $\omega_{1,2}>0$ are solved by (\ref{eq:omega12}). Furthermore\cite{Cooke.1982}, roots $\lambda=\imag\omega$ are simple and 
\begin{subequations} \label{eq:dReLambda_dTau}
\begin{align}
\tfrac{\textrm d}{\textrm d \tau} (\Re e \,\lambda)
&>0\quad&&\forall \tau \in (\tau_{1n})_{n\in \mathbb N},
\\
\tfrac{\textrm d}{\textrm d \tau} (\Re e \,\lambda)&<0&&\forall \tau \in (\tau_{2n})_{n\in \mathbb N}.
\end{align} 
\end{subequations}
The latter inequalities determine, whether roots cross the imaginary axis from left to right or from right to left, respectively, when delay is increased. Starting with stability of the delay-free system (Lemma \ref{lem:stabSwitch} (i)), it can be concluded how the number of eigenvalues in the open right half-plane $\mathbb C_{+}$ develops with increasing $\tau$.
\end{proof}
\begin{remark}[D-Subdivision Method and Stability Charts]
Substitution of $\lambda=\gamma + i \omega$ in the characteristic equation $\det \Delta(\lambda)=0$, comparing real and imaginary parts, then setting $\gamma=0$ in order to get $\omega$-parametrized curves that separate $\gamma<0$ from $\gamma>0$ in the parameter plane (so-called D-curves)  and finally analyzing the ascent direction of $\gamma$ at those curves by its partial derivative w.r.t. a system parameter, 
is known as D-subdivision method\cite{Stepan.1989}. It goes back to Neimark\cite{Neimark.1949}. 
D-curves (also called exponent-crossing curves or transition curves) separate the parameter space in regions with a constant number of unstable eigenvalues.
For the most fundamental linear RFDEs, the analytic stability criteria and corresponding stability charts in parameter planes are available in literature, see for instance \rev{Stépán\cite{Stepan.1989} or Insperger and Stépán\cite{Insperger.2011}}.
\end{remark}
Our main interest focuses on the case of prevalent delayed damping $\tilde a$. Indeed, the above described changes between being stable and unstable also occur for the equilibrium point $x_e= [0,0]^\top$ in the nonlinear system (\ref{eq:delayedSwingEq}).    
\begin{proposition}[Stability of Equilibria]
\label{prop:stabOfEquilibria}
Consider the delayed swing equation (\ref{eq:delayedSwingEq}) with $\inh\in(0,1)$.
\\
{\itshape (i)} The equilibria $x_e=[0,0]^\top+[2k\pi,0]^\top$, $k\in \mathbb Z$ are asymptotically stable in the delay-free case $\tau=0$.
\\
{\itshape (i-a) }If undelayed damping is prevalent, i.e.\ $\tilde a< a$, these equilibria remain delay-independently stable.
\\
{\itshape (i-b) }If delayed damping is prevalent, i.e.\ $a<\tilde a$, these equilibria switch stability with increasing delay. They are non-hyperbolic for delay sequences $(\tau_{1n})_{n\in \mathbb N}$, $(\tau_{2n})_{n\in \mathbb N}$ given by (\ref{eq:tau12n}) with 
\begin{align}
c=\sqrt{1-\inh^2}
\label{eq:c_inh_relation}
\end{align}
and hyperbolic otherwise. 
In the hyperbolic case, if $\tau\in\Theta_u$,
\begin{align}
\Theta_u=  \bigcup_{\substack{n\in \mathbb N\\ n \leq n_{\max}}} \!\!\!\big(\tau_{1n},\tau_{2n}\big) \;\cup \big(\tau_{1(n_{\max}+1)},\infty\big)
\label{eq:unstableTau}
\end{align}
with 
\begin{align}
n_{\max}=\max \{n\in \mathbb N: \tau_{2n}<\tau_{1(n+1)} \}
\label{eq:n_max}
\end{align}
the equilibria are unstable and if $\tau \in \mathbb R_{\geq 0} \setminus \overline \Theta_u $ 
asymptotically stable, where $\overline \Theta_u$ denotes the closure of (\ref{eq:unstableTau}).
\\
{\itshape (ii)} The equilibria in $x_e=[\pi-2y_e,0]^\top+[2k\pi,0]^\top$ are delay-independently unstable.
\end{proposition}
\begin{proof}
The results are direct consequences of the Principle of Linearized Stability (Diekmann et al.\cite{Diekmann.1995}, Chapter VII, Theorem 6.8). The linearization of the delayed swing equation (\ref{eq:delayedSwingEq}) is given by (\ref{eq:linerRFDE_delayedDamping}) with $c= \partial / \partial x_1 \sin(x_1+y_e)\vert_{x=x_e}$.
This is $c=\cos(\arcsin(\inh))=\sqrt{1-\inh^2}$ for the delay-free stable equilibria as considered in (i) and $c=-\sqrt{1-\inh^2}$ for the delay-free unstable equilibria in (ii). Stability properties of the linearized systems follow from Proposition \ref{lem:stabSwitch}.
\end{proof}

\begin{figure}
\centering
 \subfloat[parametrization (\ref{eq:param}), i.e. $\tilde a=0.0625$]{
\includegraphics[]{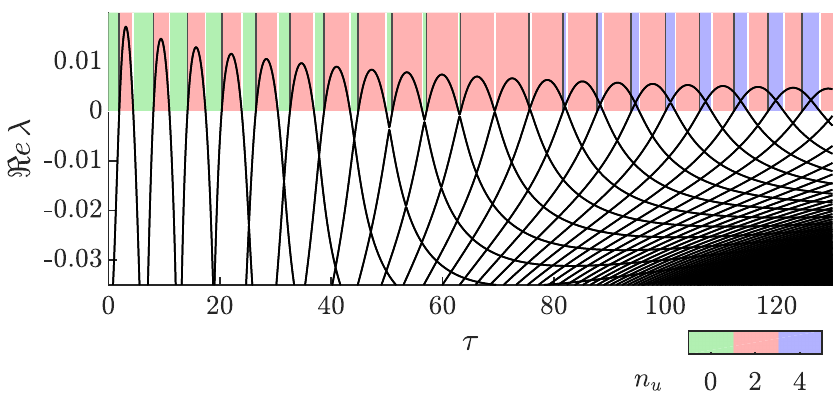}
} 
\\
 \subfloat[parametrization (\ref{eq:param}), but $\tilde a=0.225$]{
\includegraphics[]{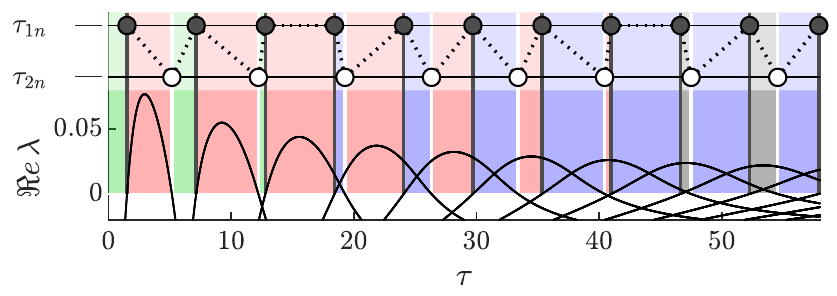}
}
\caption{\label{fig:spectrum} Numerically derived real parts of the most critical eigenvalues for different delay values: their maximum defines the spectral abscissa. Lemma \ref{lem:stabSwitch} gives two sequences of delay values $(\tau_{1n})_{n\in \mathbb N}$ (dark vertical lines) and $(\tau_{2n})_{n\in \mathbb N}$ (white vertical lines) at which the trivial equilibrium of (\ref{eq:delayedSwingEq}) is non-hyperbolic, i.e. $\exists \lambda_i: \Re e \lambda_i=0$. Colors indicate the number of unstable eigenvalues $n_u$, which changes at these points. }
\end{figure}

Figure \ref{fig:spectrum} confirms the above result (i-b) by the numerically determined spectrum for the linearization about the trivial equilibrium.
Delay values $\tau \in \mathbb R_{\geq 0 }\setminus \overline \Theta_u$, with $\Theta_u$ given in (\ref{eq:unstableTau}), correspond to stability of the equilibrium point and are depicted green. 
It should be noted that there is only a finite number of stable intervals, the last one with boundary $\tau_{1(n_{\max}+1)}$. This is due to a different rate of increase of the sequences $(\tau_{1n})_{n\in \mathbb N}$ and $(\tau_{2n})_{n\in \mathbb N}$ in (\ref{eq:tau12n}), such that according to (\ref{eq:n_max}), $\tau_{2n_{\max}}<\tau_{1(n_{\max}+1)}$ but $\tau_{1(n_{\max}+2)} \leq \tau_{2(n_{\max}+1)}$. Albeit at $\tau_{2(n_{\max}+1)}$ an eigenvalue pair indeed moves back to the left half plane, at $\tau_{1(n_{\max}+2)}$ a second eigenvalue pair has already reached the right half plane. Thus, instead of a switching between no and one unstable eigenvalue pair, henceforth a switching between one and two unstable eigenvalue pairs occurs, which in a similar manner will only last for a finite delay range. The effect becomes more obvious with a larger coefficient of delayed damping $\tilde a$: in Figure \ref{fig:spectrum} (b) the second white marker is still left to the third dark marker, i.e. $\tau_{22}<\tau_{13}$,  but $\tau_{14} < \tau_{23}$ and thus $n_{\max}=2$. 

\section{Bifurcation Analysis} \label{sec:BifurcationAnalysis}

According to Proposition \ref{prop:stabOfEquilibria}, delay values  $(\tau_{1n})_{n\in\mathbb N}$, and $(\tau_{2n})_{n\in\mathbb N}$ in (\ref{eq:tau12n})    mark points at which an eigenvalue pair crosses the imaginary axis. 
These delay values are candidates for Hopf bifurcations. 
A Hopf bifurcation goes along with the emergence of a limit cycle, meaning that either for some smaller or for some larger delay values a limit cycle surrounds the equilibrium point. 
\subsection{Analytic Equilibrium Bifurcation Analysis}
There are different indicators in literature which describe the character of a Hopf bifurcation. In the following we use the value of the first Lyapunov coefficient in the bifurcation point $L:=l_1(0)$. For precise Hopf theorems in time delay systems we refer to Diekmann et al.\cite{Diekmann.1995} and Hale and Verduyn Lunel\cite{Hale.1993}.
\begin{remark}[Relation between the First Lyapunov Coefficient and other Indicators] 
For the relation to the coefficient of the resonant cubic term in the Poincaré normal form see Kuznetsov\cite{Kuznetsov.1998}. For the relation to the second order coefficient in a series expansion of the Floquet exponent see Hassard et al.\cite{Hassard.1981} and Kazarinoff et al.\cite{Kazarinoff.1978}. It should be noted that all three mentioned numbers 
obtain the same sign, which is actually the decisive value. 
\end{remark}
Indeed, to mark a Hopf bifurcation is even a generic property of stability switching points with $\omega\neq 0$, since they have already proven to obtain a pair of complex conjugate roots $\lambda=\pm \imag \omega \neq 0$ which crosses the imaginary axis with increasing delay (transversality condition). 
Assume there is no other pair of roots simultaneously on the imaginary axis.
The only circumstance that could hinder these stability switching points from being Hopf points is a vanishing first Lyapunov coefficient $L=0$. 
The latter is referred to as degenerate case\cite{Kuznetsov.1998}. The linearized equation is an example of degeneracy. It is the prototype system of stability switching. However, 
obviously it cannot give rise to a limit cycle (which requires the closed orbit to be isolated).
However, not only whether the Lyapunov coefficient is zero is of interest.
The sign of the Lyapunov coefficient is responsible for the type of the Hopf bifurcation. 
\begin{definition}[Sub- and Supercritical Hopf Bifurcation\cite{Kuznetsov.1998}]
Let $L$ denote the first Lyapunov coefficient in a Hopf bifurcation point. The Hopf bifurcation is called 
{\itshape supercritical}, if $L<0$ and 
{\itshape subcritical}, if $L>0$. 
\end{definition}
The term Hopf bifurcation is sometimes restricted to switches of the equilibrium between being stable and unstable as described above (obviously, in a system with only two eigenvalues, e.g., in the restriction to the two-dimensional center manifold, an eigenvalue pair cannot  cross the imaginary axis without such a stability switch). Then, the sign of the first Lyapunov coefficient has the following consequences:
\begin{lemma}[Sub- and Supercritical Hopf Bifurcation\cite{Kuznetsov.1998}]
\label{lem:SubAndSuperConsequences}
Consider a Hopf bifurcation at a stability switching point. The following implications are valid in a neighborhood of the bifurcation point.\\
{\itshape (i)} $L<0$ (supercritical) $\Leftrightarrow$
a stable limit cycle occurs for bifurcation parameter values at which the equilibrium is unstable.
\\
{\itshape (ii)}  $L>0$ (subcritical) $\Leftrightarrow$ 
an unstable limit cycle occurs for bifurcation parameter values at which the equilibrium is stable.
\end{lemma}
Thereby, a stability analysis - be it analytically (Lemma \ref{lem:stabSwitch}) or numerically (Figure \ref{fig:spectrum}) - already reveals whether stability of the equilibrium point is given on the side of smaller delay values or on the side of larger delay values. 
\newpar
 However, unstable eigenvalues can already exist when an eigenvalue pair crosses the imaginary axis. Consequently, there might be, besides of the center and stable manifold, also an unstable manifold in the bifurcation point.
 The equilibrium changes from being of type $k$, i.e.\ having $k$ eigenvalues on the right half-plane, to type $(k+2)$ or vice versa. Then, in analogy to Lemma \ref{lem:SubAndSuperConsequences}, a Hopf bifurcation can lead to the emergence of 
 a limit cycle with $k$ or $(k+1)$ unstable Floquet multipliers. 
While system (\ref{eq:main}) has, dependent on the parametrization, an arbitrary but finite number\cite{Cooke.1982} of stability switching points $\tau_{1n}$, $n\leq n_{\max}+1$, $\tau_{2n}$, $n\leq n_{\max}$, there is actually an infinite number of non-hyperbolic points $\tau_{1n}$, $\tau_{2n}$, $n\in \mathbb N$ at which a pair of complex eigenvalues crosses the imaginary axis. All delay values in both sequences are candidates for Hopf bifurcations in this extended sense.
\newpar
In the following Theorem, we give a result for the sign of the first Lyapunov coefficient which is decisive according to Lemma \ref{lem:SubAndSuperConsequences}. We consider the general  case of second order differential equations with delayed damping and arbitrary, delay-free nonlinearity.
\begin{theorem}[First Lyapunov Coefficient for a General Class of Systems]
\label{th:LyapCoeffRes}
Consider the RFDE with delayed damping and delay- and derivative-free nonlinearity
\begin{align}
\ddot y(t) + \tilde a \dot y(t-\tau) + a \dot y(t) + h(y(t))=0,
\label{eq:DelayedDampingGeneralNonlinearity}
\end{align}
$y(t)\in \mathbb R$, $\tilde a, a, \tau \in \mathbb R_{>0}$, $a<\tilde a$, $h\in C^\omega(\mathbb R, \mathbb R)$. Let $y_e\in \mathbb R$ such that $h(y_e)=0$ and $h_e':=\frac{\partial h}{\partial y} \vert_{y_e}>0$ and
\begin{align}
\tau\in (\{\tau_{1n}\}_{n\in \mathbb N} \cup \{\tau_{2n}\}_{n\in \mathbb N}) \setminus (\{\tau_{1n}\}_{n\in \mathbb N} \cap \{\tau_{2n}\}_{n\in \mathbb N})
\end{align}
 with $\tau_{jn}$ given by (\ref{eq:tau12n}) and $c:=h_e'$. Then, the sign of the first Lyapunov coefficient is determined by
\begin{align}
\textrm{sgn } L = \textrm{sgn} \Bigg[& \Re e \left(\frac{1}{\beta \det \Delta (2\imag\omega)} \right)\nonumber \\
&+ 
\Re e \left(\frac 1 {\beta} \right)
\left( \frac 2 {h_e'} -\frac {h'''_e}{(h''_e)^2}\right)
\Bigg] 
\label{eq:LyapCoeffDelayedDampingGeneralNonlinearity}
\end{align}
with
\begin{align}
 \beta=&-(\omega \tau )(\omega^2-h_e') + \imag ((\omega \tau) \omega a +h_e'+\omega^2) 
\label{eq:beta_in_sgnL}
\end{align}
and
\begin{align}
\det \Delta (2\imag\omega)=& -3\omega^2-\left(1+4\frac a{\tilde a} \right)(\omega^2-h_e') \nonumber\\
& +2\imag \omega \left( a- \tilde a + 2\frac {a^2}{\tilde a}  \right),
\end{align}
where $h_e',h_e'',h_e'''$ denote respective derivatives of $h$ evaluated in the equilibrium point $y_e$.
\end{theorem} 
\begin{proof}
The proof can be found in the Appendix.
\end{proof}
It should be noted, that, due to (\ref{eq:beta_in_sgnL}), in the second summand of (\ref{eq:LyapCoeffDelayedDampingGeneralNonlinearity}) the term
\begin{align}
\textrm{sgn} \; \Re e \left(\frac 1 {\beta}\right)=\textrm{sgn}\left(\frac {\Re e (\beta)}{\vert \beta\vert^2}\right)= -\textrm{sgn} (\omega^2-c)
\end{align}
with $c:=h_e'$ is decisive. 
The latter depends on whether the bifurcation point is an element of $(\tau_{1n})_{n\in\mathbb N}$ (hence $\omega=\omega_1$) or collected in $(\tau_{2n})_{n\in\mathbb N}$ (hence $\omega=\omega_2$), since 
 according to (\ref{eq:compRe})
$
\omega_1^2-c >0$ and $\omega_2^2-c <0$.
 Consequently, 
\begin{align}
\textrm{sgn} \; \Re e \left(\frac 1 {\beta}\right) =\left\{\begin{array}{ll} -1, & \tau \in (\tau_{1n})_{n\in \mathbb N} \\
+1 , &\tau \in (\tau_{2n})_{n\in \mathbb N}.  \end{array}\right.
\label{eq:sgn_Re_beta}
\end{align}
\begin{proposition}
\label{prop:typeHopf}
Consider the delayed swing equation (\ref{eq:main}), (\ref{eq:param}). 
Let $a<\tilde a$ and $\tau\in (\{\tau_{1n}\}_{n\in \mathbb N} \cup \{\tau_{2n}\}_{n\in \mathbb N}) \setminus (\{\tau_{1n}\}_{n\in \mathbb N} \cap \{\tau_{2n}\}_{n\in \mathbb N})$  with $\tau_{jn}$ given by (\ref{eq:tau12n}) and $c:=\sqrt{1-\inh^2}$. There occur
\begin{itemize}
\item {\itshape supercritical Hopf} bifurcations ($L<0$) at $\tau\in (\tau_{1n})_{n\in \mathbb N}$,
\item {\itshape subcritical Hopf} bifurcations ($L>0$) at  $\tau \in (\tau_{2n})_{n\in \mathbb N}$. 
\end{itemize}
\end{proposition}
\begin{proof}
The delayed swing equation (\ref{eq:main}) is equivalent to system (\ref{eq:DelayedDampingGeneralNonlinearity}) with $h(y)=\sin(y)-\inh$. The 
equilibrium of interest is $y_{e}=\arcsin(\inh)$ and derivatives in this point become
$c=h'_e=\cos y_{e}=\sqrt{1-\inh^2}$,
$h''_e=-\sin y_{e}=-\inh$, 
$h'''_e=-\cos y_{e}=-\sqrt{1-\inh^2}$. 
For the chosen parameters (\ref{eq:param}), the second summand in (\ref{eq:LyapCoeffDelayedDampingGeneralNonlinearity})  is dominant and thus (\ref{eq:sgn_Re_beta}) is decisive. To be more precise, a numerical evaluation of (\ref{eq:LyapCoeffDelayedDampingGeneralNonlinearity}) gives
\begin{align}
\textrm{sgn}\;L= \left\{
\begin{array}{ll}
\textrm{sgn} \tfrac{(0.692n+0.260)+(-149.155n-47.057)}{n^2+4.691n+27.137},& \tau\in (\tau_{1n})_{n\in\mathbb N} 
\\
\textrm{sgn} \tfrac{(-0.899n-0.552)+(157.982n+108.140)}{n^2+5.429n+29.004},\!\!& \tau\in (\tau_{2n})_{n\in\mathbb N} .\end{array}
 \right.
\nonumber
\end{align}
\end{proof}

Hence, where elements in  $(\tau_{1n})_{n\in\mathbb N}$ alternate with elements of  $(\tau_{2n})_{n\in\mathbb N}$, sub- and supercritical Hopf bifurcations alternate as well. Consequently, all bifurcations are equally orientated in the following sense.
\begin{corollary}[Direction of Limit Cycle Continuation]
\label{cor:DirLimitCycles}
For each Hopf bifurcation point $\tau_{jn}$ in  (\ref{eq:main}) with (\ref{eq:param}),  existence and increase of the emerging limit cycle is locally ensured on the side of larger delay values $\tau>\tau_{jn}$.
\end{corollary}
\begin{proof}
Due to (\ref{eq:dReLambda_dTau}), the eigenvalue pair crosses the imaginary axis from left to right if $\tau\in (\tau_{1n})_{n\in \mathbb N}$. Hence, a larger delay value goes along with two more unstable eigenvalues. A Hopf bifurcation with negative first Lyapunov coefficient $L<0$ results in a limit cycle on this side (Lemma \ref{lem:SubAndSuperConsequences} (i)). In contrast, if $\tau\in (\tau_{2n})_{n\in \mathbb N}$, eigenvalues cross from right to left. Thus, the side of larger delays goes along with two more stable eigenvalues. However, since the sign of the first Lyapunov coefficient is also reversed, i.e.\ $L>0$ at $\tau\in (\tau_{2n})_{n\in \mathbb N}$, the emerging limit cycle occurs again for larger delay values (Lemma \ref{lem:SubAndSuperConsequences} (ii)). 
\end{proof}

\begin{figure}
\centering
 \subfloat[]{
\includegraphics[]{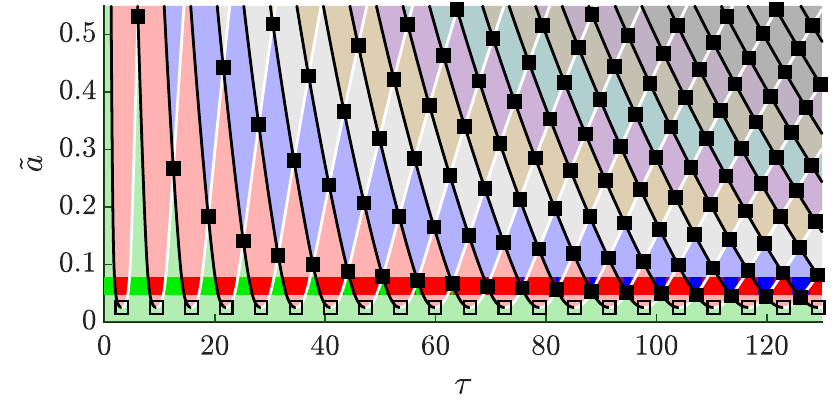}
} 
\\
 \subfloat[detail]{
\includegraphics[]{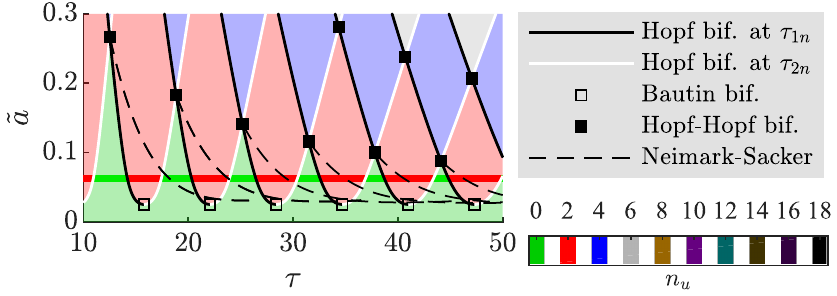}
}
\caption{\label{fig:codim2}Colors indicate the analytically derived number $n_u$ of unstable equilibrium eigenvalues in the $(\tau,\tilde a)$ parameter plane (other parameters according to (\ref{eq:param})). The separating black and white lines mark Hopf points in $(\tau_{1n})_{n\in \mathbb N}$ and $(\tau_{2n})_{n\in \mathbb N}$, respectively (cmp. vertical lines in Figure \ref{fig:spectrum}). At their intersection points, Hopf-Hopf bifurcations occur. These Hopf-Hopf equilibrium bifurcation points form starting points of Neimark-Sacker bifurcation curves of limit cycles, at which invariant tori emerge. The numerically derived dashed lines in (b) refer to those Neimark-Sacker bifurcations which have been detected in Figure \ref{fig:overview} with $\tilde a=0.0625$ (non-transparent line in (a) and (b)).}
\end{figure}

Figure \ref{fig:codim2} is based on the above analytic results for (\ref{eq:main}), (\ref{eq:param}). It considers the number of unstable eigenvalues not only dependent on $\tau$ but also on a variation of $\tilde a$. Positions of bifurcation delay values (\ref{eq:tau12n}) change if another parameter like $\tilde a$ is varied. Thereby, possible points of codimension two bifurcations become visible. However, there are some further non-degeneracy conditions for these types of bifurcations. According to the numerical bifurcation tool \texttt{DDE-BIFTOOL}\cite{Engelborghs.2002} these conditions indeed hold in the points we have examined.
\begin{itemize}
 \item {\itshape Hopf-Hopf} bifurcations (also called two-pair bifuraction, double-Hopf bifurcation) in the $(\tau,\tilde a)$ plane or in the  $(\tau,\inh)$ plane become possible if $a$, $\tilde a$ and $\inh$ are chosen such that $a<\tilde a$ and $\tau \in  (\{\tau_{1n}\}_{n\in \mathbb N} \cap \{\tau_{2n}\}_{n\in \mathbb N})$. Such intersection points of $\tau_{1n}$ and $\tau_{2n}$ curves in the $(\tau,\tilde a)$ plane are marked in Figure \ref{fig:codim2}. At these parameter combinations, two pairs of purely imaginary eigenvalues $\pm \imag \omega_1$, $\pm \imag \omega_2$ cross simultaneously the imaginary axis in opposite directions. 
\item {\itshape Bautin} bifurcations (also called degenerate Hopf bifurcation, generalized Hopf bifurcation) in the $(\tau,\tilde a)$ plane are possible if $a=\tilde a$.  Then, at $\tau_{n}=(1-w^2)^{-1/2} \pi (2n+1)$, $n\in \mathbb N$ an eigenvalue pair lies on the imaginary axis, but the Lyapunov coefficient vanishes (Figure \ref{fig:codim2}). 
\item {\itshape Fold-Hopf} bifurcations (also called zero-Hopf, saddle-node Hopf, Gavrilov-Guckenheimer bifurcation) in the $(\tau,\inh)$ plane  become possible if $\inh=1$ (cmp. top view of Figure \ref{fig:overview}) and $\tilde a>a$. Due to (\ref{eq:c_inh_relation}) $c=0$. Hence, $\lambda=0$ becomes an eigenvalue (cmp. (\ref{eq:lambda_neq_0})), while at $\{\tau_{1n}\}_{n\in \mathbb N}$  with $\omega_{1}=\sqrt{\tilde a ^2-a^2}$ an eigenvalue pair is simultaneously on the imaginary axis. 
\end{itemize}
Furthermore, in Figure \ref{fig:codim2} Hopf bifurcation curves bound the green marked area of stability. Such Christmas tree like stability charts manifest the possibility of stability switching and are well known in time delay systems\cite{Boese.1989, Abdallah.1993}. 

\subsection{Numeric Approximations of Emerging Limit Cycles}
\begin{figure*}
\includegraphics[]{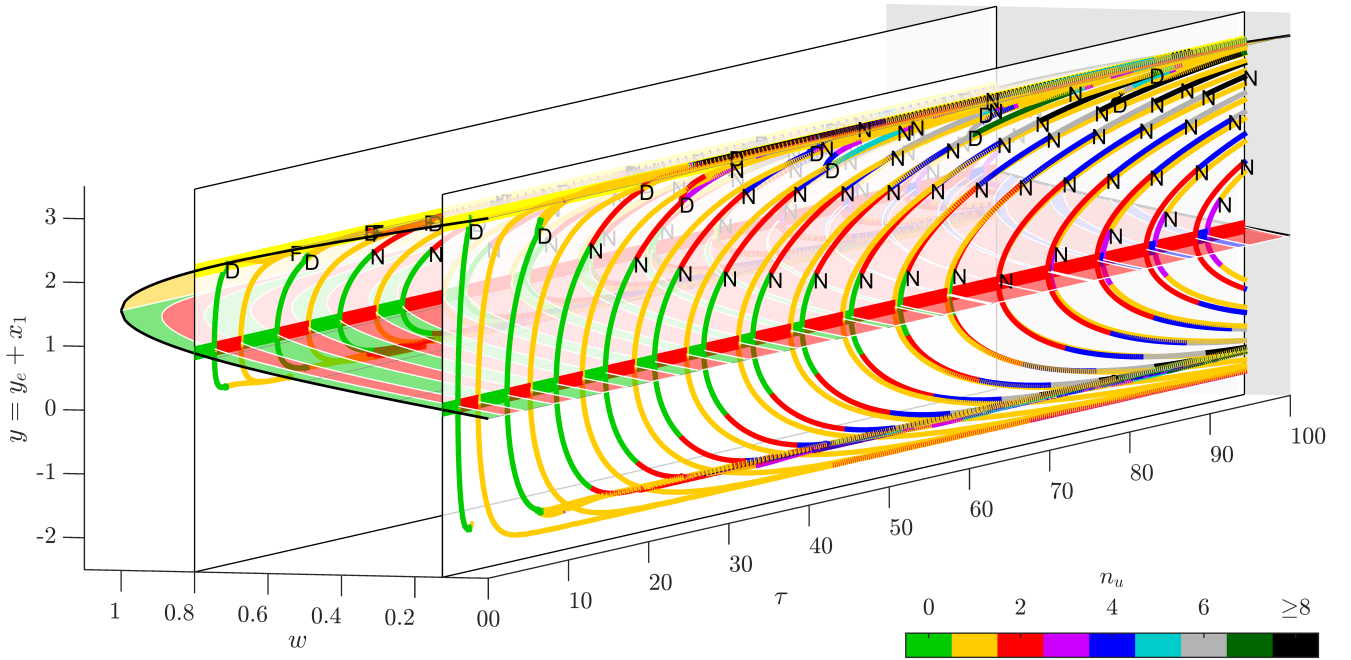} 
\caption{\label{fig:overview}The bended surface corresponds to $y=y_e+x_1$ coordinates of equilibria over the $(\tau,\inh)$ parameter plane. Thereby, colors indicate the number of unstable eigenvalues. Additionally, for fixed parameter values $\inh\in\{0.125,0.8\}$, bold, striped horizontal lines are drawn along the lower equilibrium surface (the first one corresponds to Figure \ref{fig:spectrum} (a) and can also be found in Figure \ref{fig:codim2}). They emphasize that the number of unstable eigenvalues changes with increasing delay. Where Hopf bifurcations occur along such lines, limit cycles emerge. On the vertical, white $(\tau,y)$ planes, maximal and minimal $x_1$ coordinates of these limit cycle solutions are given. Line colors are chosen according to the number of unstable Floquet-Multipliers. This limit cycle representation equals the outline of a top view in the representation of Figure \ref{fig:hom} with different $\tau$-scaling. Letters mark detected limit cycle bifurcations like Neimark-Sacker (N), period doubling (D) or limit cycle fold (F).}
\end{figure*}
A numerical analysis does not only determine normal form coefficients and thus the type of bifurcation. It also allows to track limit cycles which emerge from Hopf bifurcations and to approximate the corresponding periodic solution. The MATLAB Toolbox \texttt{DDE-BIFTOOL}\cite{Engelborghs.2002} is a valuable tool for this purpose. Figure \ref{fig:overview} visualizes our results. 
\newpar
The colored surface in Figure \ref{fig:overview} describes the equilibrium manifold in the $(\tau,\inh,x_1)$ space. Assume parameter values $\tau$ and $\inh$ are given, then the $x_1$ coordinates of equilibria can be read off. Stability ranges, which are indicated by color, are due to analytic results (Proposition \ref{prop:stabOfEquilibria}). The upper equilibrium branch is delay independently unstable (one unstable eigenvalue), while the lower branch switches its stability property with increasing delay. For $\inh>1$ no equilibrium exists. Not only in the delay-free system, but also for $\tau>0$, the upper and lower equilibrium branch collide in a fold bifurcation at $\inh=1$. We consider an equilibrium branch along $[\tau,\inh,x_1]^\top=[\tau,0.125,0]^\top$ with increasing $\tau$-values. At the Hopf bifurcation points $\tau\in(\tau_{in})_{n\in\mathbb N}$, $i\in\{1,2\}$ a limit cycle branches off. Maximum and minimum $x_1$ values of the corresponding periodic solutions are plotted. 
\newpar
Thereby, the meaning of alternating super- and subcritical Hopf bifurcations (Proposition \ref{prop:typeHopf}) becomes clear. For $\tau<\tau_{1,1}$ the equilibrium in $[\tau,\inh,x_1]^\top=[\tau,0.125,0]^\top$ is stable. For $\tau_{1,1}\in (\tau_{1n})_{n\in\mathbb N}$, $n\leq n_{\max}+1$ stability is lost and a supercritical Hopf bifurcation is responsible for the occurence of a stable limit cycle at larger delay values. The limit cycle surrounds the henceforth unstable equilibrium. Of course, only locally, i.e.\ for delay values nearby the bifurcation point, stability and even existence of the limit cycle are granted. At parameter values $\tau_{2,1}\in (\tau_{2n})_{n\in\mathbb N}$, $n\leq n_{\max}$ the equilibrium regains stability. Again, a Hopf bifurcation causes the branch off of a limit cycle. However, this time, it does not surround the unstable equilibrium which would mean the limit cycle exists for delay values smaller than $\tau_{2,1}$. Instead, the limit cycle occurs again for larger delay values (Corollary \ref{cor:DirLimitCycles}) - now surrounding the stable equilibrium as result of a subcritical Hopf bifurcation. This pattern will repeat for larger delay values. The number of detected coexisting limit cycles rises with increasing delay.
\begin{remark}[Consequences for the Domain of Attraction]
Domains of attraction are essential in power system stability analysis.
Results of Schaefer et al.\cite{Schafer.2015} suggest that for certain delay ranges an increasing time delay is able to enlarge (in the sense of a generalized basin stability\cite{Menck.2013,Leng.2016}) the domain of attraction of the delay-free stable equilibrium point. This seems to be in accordance with the above detected growing limit cycles. Indeed, in second order ODEs a limit cycle which stems from a Hopf bifurcation would clearly mark the boundary of the domain of attraction for parameter values nearby. Hence, the detected growing limit cycles would go along with a growing domain of attraction as the bifurcation parameter increases.
However, infinite dimensionality of time delay systems makes meaningful conclusions of this manner hardly possible. There is no reason why solutions which are closer to the attractor should not diverge and cross the limit cycle in the $(x_1,x_2)$ plane, i.e.\ in the $(x_{1t}(0),x_{2t}(0))$ projection of the infinite dimensional state space.  Nevertheless, the results enable us to gain upper bounds on the radius of attraction. Such upper bounds are needed in order to prove that a minimum norm requirement on the domain of attraction is not fulfilled. We refer to Scholl et al.\cite{Scholl.2019}, Example VI.1.
\end{remark}
\subsection{Numeric Results on Bifurcations of Limit Cycles}
	\begin{figure}
\includegraphics[]{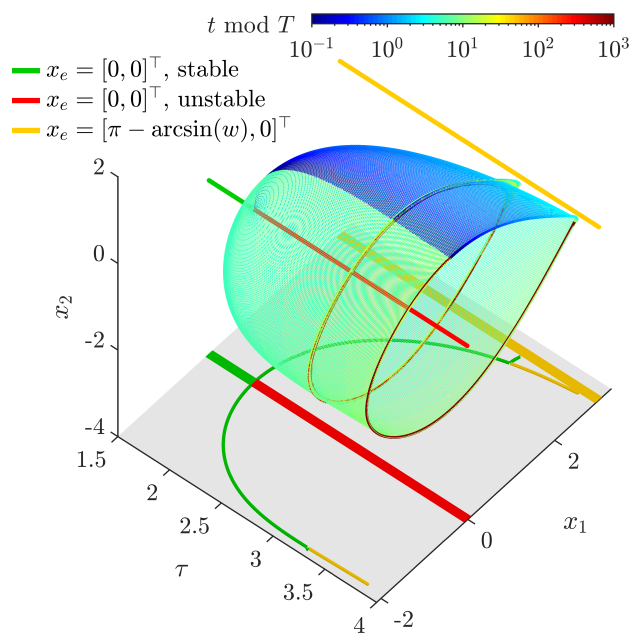}
\caption{\label{fig:hom} For a given delay value $\tau$, a vertical slice shows the appearance of the limit cycle in the $(x_1,x_2)$ plane. Color indicates the elapsed time, if the limit cycle solution segment with $\phi_1(0)=0$, $\phi_2(0)\geq 0$ is chosen as initial function $\phi$. At the bottom, the extreme $x_1$ values are sketched. These correspond to the first arc at $\inh=0.125$ in Figure \ref{fig:overview}. The limit cycle emerges at $\tau= 1.93$ due to a supercritical Hopf bifurcation. At $\tau=3.26$ it looses stability in a period doubling bifurcation. A stable limit cycle branches off. The original limit cycle as well as its period doubling branch off grow and evolve to homoclinic orbits of the type-one saddle equilibrium at $\tau=3.29$ and $\tau=3.83$, respectively. }
\end{figure}
The Hopf theorem does only guarantee stability properties and existence of a limit cycle locally in a small one-sided neighborhood of the bifurcation point. However, in limit cycle bifurcations, stability and even existence of limit cycles can get lost for larger parameter variations.
\newpar 
Based on a two-point boundary value problem, with boundary condition $x_0=x_T$ formulating periodicity, a numeric approximation of the $T$-periodic solution as well as the unknown period $T$ can be derived\cite{Roose.2007}. Figure \ref{fig:hom} visualizes how the limit cycle, which emerges in the first Hopf bifurcation of (\ref{eq:delayedSwingEq}),(\ref{eq:param}) at $\tau=\tau_{11}$, develops with increasing delay $\tau$: for each $\tau$ value, the corresponding vertical slice shows the limit cycle in the $(x_1,x_2)$ plane. 
\begin{itemize}
	\item A {\itshape homoclinic bifurcation} becomes visible as the growing limit cycle evolves to a homoclinic orbit of the adjacent equilibrium, i.e.\ $T\to \infty$ (Figure \ref{fig:hom}).

	\end{itemize}
Besides the location, numeric approximations of the Floquet multipliers
allow to conclude stability of limit cycles.
As in ODEs, there is always the trivial Floquet multiplier at $\mu=1+0\imag$, which does not affect stability of the limit cycle\cite{Breda.2015}. 
However, in the course of a bifurcation parameter variation, a further Floquet multiplier might come to be at the the stability boundary, the unit cycle in the $\mathbb C$ plane. 
Again this cannot happen without a major qualitative change in the system behavior, i.e.\ a bifurcation. The following well known types of limit cycle bifurcations have been detected in the delayed swing equation (Figure \ref{fig:overview}):   
\begin{itemize}
	\item In a {\itshape limit cycle fold bifurcation (tangent bifurcation, saddle-node bifurcation, limit point)} a stable and an unstable limit cycle collide, i.e.\ at the bifurcation point, there remains only one limit cycle. Thereby, a (further, besides the trivial one) Floquet multiplier is on the unit cycle at $\mu_{c}=1+0\imag$. By larger parameter variations the limit cycle disappears at all. On the back plane with $\inh=0.8$ in Figure \ref{fig:overview}, such a fold bifurcation is visible at the point where the first two limit cycles  collide. 
	\item {\itshape Period doubling bifurcations (flip bifurcations)} go along with a crossing of the unit cycle at the opposite side: $\mu_{c}=-1+0\imag$. A new limit cycle with double period branches off. Figure \ref{fig:PD} shows a sequence of period doubling bifurcations. \rev{Delay values and limit cycle periods of consecutive period doubling bifurcations are marked by black squares in Figure \ref{fig:PD_tau_T}.} The linear increase \rev{of these points} in the double logarithmic diagram does not only show the doublings of limit cycle periods (i.e.\ the solution over two periods does no longer consist of two equal solution segments - thus twice as many encirclings of the equilibrium are needed to describe the limit cycle).  It also reveals the geometric progression of bifurcation points: \rev{an accelerating succession of bifurcations can be observed as a certain delay value $\tau_\infty$ is approached. The horizontal axis of Figure \ref{fig:PD_tau_T} describes time delay $\tau$ in terms of the difference to this reference value  $\tau_\infty$. Linearity in the logarithmic representation uncovers the exponential behavior.}  It is well known, that period doubling cascades form a possible route to chaos. Whether the cascade has indeed infinitely many elements and hence leads to a non-periodic ($T=\infty$) limit set \rev{at $\tau=\tau_\infty$}, cannot be inferred from our results.
	
	\item {\itshape A Neimark-Sacker bifurcation (torus bifurcation)} is the counterpart to the Hopf bifurcation of an equilibrium. It occurs if a complex conjugate Floquet multiplier pair crosses the unit cycle at $\textrm{e}^{\pm i\varphi}$ with  $\textrm{e}^{ik\varphi}\neq 1+0\imag$ for $k\in\{1,2,3,4\}$. While in a Hopf bifurcation of an equilibrium a limit cycle evolves which encircles the equilibrium, by a Neimark-Sacker bifurcation an invariant torus around the limit cycle might result. How the position of Neimark-Sacker bifurcation points in Figure \ref{fig:overview} changes if $\tilde a$ varies, is given in Figure \ref{fig:codim2} (b).
\end{itemize}

\section{Conclusions}

When time delay is increased, linear second order systems with additional delayed damping or restoring force are able to show again and again characteristic roots crossing the imaginary axis. In nonlinearily perturbed systems each crossing is a candidate for a Hopf bifurcation point, which is further characterized by the sign of the first Lyapunov coefficient. 
The present paper gives a formula of the latter for general second order differential equations with additional delayed damping and arbitrary, delay-free nonlinearity. The result is applied to the delayed swing equation. It allows to prove, that sub- and supercritical Hopf bifurcations are combined in such a way that all emerging limit cycle branches head towards the direction of larger delays. As a consequence, more and more limit cycles are likely to coexist with increasing delay. These limit cycles are tracked by numeric means and respective graphical representations are given.
The latter show how due to further limit cycle bifurcations a surprisingly rich variety of complicated dynamics results - including invariant tori and period doubling cascades.  
Summarizing, bifurcations of the following types are realized: Hopf, Bautin, Hopf-Hopf, fold, limit cycle fold, homoclinic, period doubling and Neimark-Sacker. 
Hence, time delay already generates in the simple swing equation phenomena which are only known from higher order power system models.
\begin{figure}
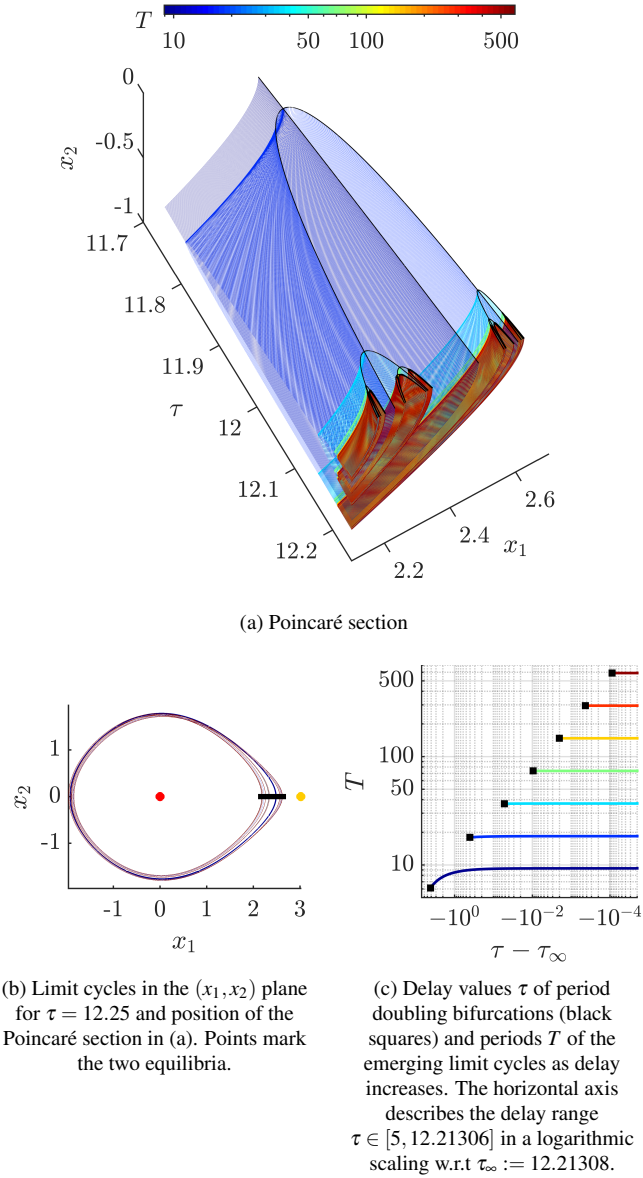

\centering
 \subfloat[Poincaré section]{
\includegraphics[]{fig_PD_px.pdf}
} 
\\
 \subfloat[Limit cycles in the $(x_1,x_2)$ plane for $\tau=12.25$ and position of the Poincaré section in (a). Points mark the two equilibria.]{
\includegraphics[]{fig_PD_2d_px.pdf}
}
\hfill
 \subfloat[\label{fig:PD_tau_T}\rev{Delay values $\tau$ of period doubling bifurcations (black squares) and periods $T$ of the emerging limit cycles as delay increases. 
The horizontal axis describes the delay range $\tau\in [5,12.21306]$ in a logarithmic scaling w.r.t\ $\tau_\infty:=12.21308$.}]{
\includegraphics[]{fig_PD_period_px_new.pdf}
}
\caption{\label{fig:PD} \revv{A finite dimensional Poincaré section} (as black marked cutting edges) of a period doubling cascade is shown in (a). Depicted is a horizontal cut through limit cycles from 6 subsequent period doubling bifurcations. In contrast to simulatively derived Poincaré \revv{sections}, the unstable limit cycles remain visible. The perspective in (a) is similar to Figure \ref{fig:hom} albeit $\tau\in[11.7,12.25]$, $x_2<0$, $x_1>2.1$. This time, extreme $x_1$ values of the full limit cycles refer to the second green arc at $\inh=0.125$ in Figure \ref{fig:overview}.   }
\end{figure}


\begin{acknowledgments}
The authors acknowledge support by the Helmholtz Association under the Joint Initiative "Energy Systems Integration" (ZT-0002).
\end{acknowledgments}


\section{Appendix}
In order to prove Theorem \ref{th:LyapCoeffRes}, first the following Lemma is recalled in the general case.
\begin{lemma}[First Lyapunov Coefficient\cite{Bosschaert.2015,Wage.2014}] \label{lem:firstLyapCoeffGeneral}
Consider a nonlinear RFDE with $m$ discrete delays 
\begin{align}
\dot x(t)=& f(X), \label{eq:FirstLyap_system}\\
& X=[x(t),x(t-\tau_1),\ldots,x(t-\tau_m)]\in \revv{\mathbb R^{n\times (m+1)} ,} \nonumber
\end{align}
$x(t)\in \mathbb R^n$ 
, $f \in C^\omega(\revv{\mathbb R^{n\times (m+1)}}, \mathbb R^n)$ , 
$f(0_{n\times (m+1)})=\revv{0_n}$ 
in a Hopf bifurcation point with Taylor expansion about $X_e=0_{n\times (m+1)}$ 
\begin{align}
\dot x(t)=& \textrm D f(0_{n\times (m+1)})(X) + \frac 1 2 \textrm D^2 f(0_{n\times (m+1)})(X,X) \nonumber \\
& 
+ \frac 1 {3!} \textrm D^3 f(0_{n\times (m+1)})(X,X,X) +\mathcal O(X^4).
\label{eq:FirstLyap_systemTaylor}
\end{align} 
Assume the simple pair of purely imaginary eigenvalues is $\lambda_{1,2}=\pm \imag \omega$,  given as roots of the characteristic equation $\det \Delta(\lambda)=0$. Let $p,q\in \mathbb R^n$ fulfill
\begin{align}
p^\top\Delta(\imag \omega)=0, \quad \Delta(\imag \omega)q=0 \quad \textrm{ with } p^\top \textrm D \Delta (\imag \omega) q =1.
\label{eq:pq_normalization}
\end{align}
Assign
\begin{align}
M\colon C([-\tau,0],\mathbb R^n) &\to \revv{\mathbb R^{n\times (m+1)}} , \nonumber \\
\phi &\mapsto M(\phi):=[\phi(0),\phi(-\tau_1),\ldots,\phi(-\tau_m)], \nonumber
\end{align}
i.e.\ $X=M(x_t)$.
Then, the first Lyapunov coefficient is given by
\begin{align}
L=\frac{1}{\omega}\Re e \Big(p^\top\Big[ 
& \textrm D^2 f(0_{n\times (m+1)})( {M(\phi)}, M(h_{11}))
\nonumber\\
+& \frac 1 2 \textrm D^2 f(0_{n\times (m+1)})( {M(\overline \phi)}, M(h_{20})) 
\nonumber\\
+& \frac 1 2 \textrm D^3 f(0_{n\times (m+1)})(  {M(\phi)}, {M(\phi)},{M(\overline \phi)})
\Big]\Big),
\label{eq:LyapCoeffGeneral}
\end{align}
where
\begin{align}
\phi(\theta) &=\textrm{e}^{\imag\omega\theta}q  \nonumber\\
 h_{11}(\theta) &= \Delta^{-1}\!(0) \;\textrm D^2 f(0_{n\times (m+1)})({M(\phi)},{M(\overline \phi)}) \nonumber \\
 h_{20}(\theta)&= \textrm{e}^{2\imag \omega\theta} \Delta^{-1}\!(2\imag \omega) \;\textrm D^2 f(0_{n\times (m+1)})({M(\phi)},{M(\phi)}). 
\label{eq:phi_h11_h20}
\end{align}
\end{lemma}
Based on Lemma \ref{lem:firstLyapCoeffGeneral}, we are able to state the proof of Theorem \ref{th:LyapCoeffRes}. 
\begin{proof} [Proof of Theorem \ref{th:LyapCoeffRes}]
Consider the state space representation of (\ref{eq:DelayedDampingGeneralNonlinearity}) in 
$x(t)=[\xI (t),\xII (t)]^\top$
\begin{subequations}
\begin{align}
\dot \xI (t)&=\xII (t)\\
\dot \xII (t) &= -a\xII (t)-\tilde a \xII (t-\tau) -h(\xI (t))
\end{align}
\label{eq:stateSpaceGeneral}
\end{subequations}
with Taylor expansion about $x_e=[0,0]^\top$
\begin{align}
\dot x(t)=&
\begin{bmatrix} 0 & 1  \\ -h_e' & -a\end{bmatrix}
x(t) 
+
\begin{bmatrix} 0 & 0  \\  0& -\tilde a \end{bmatrix} 
x(t-\tau)
\nonumber \\
&+\begin{bmatrix}
0 \\
-\frac 1 2 h_e''\xI^2 (t)-\frac 1 {3!} h_e'''\xI^3 (t) + \mathcal O(\xI^4)
\end{bmatrix}.
\label{eq:TaylorStateSpaceGeneral}
\end{align}
where $h_e^{(k)}$ denotes the $k$-th derivative of $h(\xI)$ in $\xI =0$. In order to evaluate (\ref{eq:pq_normalization}), the characteristic matrix given by (\ref{eq:Delta}) with $c:=h_e'$ is considered. Due to  (\ref{eq:iwaexp}), an evaluation at $\lambda=\imag \omega$ can be simplified by 
$
a+\tilde a \textrm e^{-\imag \omega\tau} 
=-\imag \omega +\imag \frac c \omega.
$
It follows
\begin{align}
\Delta(\imag \omega)=\begin{bmatrix}\imag \omega & -1 \\ c & \imag \frac c \omega \end{bmatrix}
\end{align}
with 
$
 \ker \Delta^\top(\imag \omega)=\begin{bmatrix} \imag c & \omega\end{bmatrix} ^\top \mathbb C
$
and
$
\ker \Delta(\imag \omega)=\begin{bmatrix} 1 & \imag \omega \end{bmatrix}^\top \mathbb C 
$.
According to (\ref{eq:pq_normalization}) normalization factors  $\alpha_p,\alpha_q\in \mathbb C$, have to be chosen, such that 
\begin{align}
p^\top:=\alpha_p  \begin{bmatrix} \imag c & \omega\end{bmatrix} ^\top ,\quad 
q:=\alpha_q  \begin{bmatrix} 1 \\ \imag \omega\end{bmatrix}  
\label{eq:pq}
\end{align}
fulfill a relative normalization $p^\top \textrm D \Delta(i\omega) q\stackrel ! =1$, where the derivative of the characteristic matrix is given by 
\begin{align}
\textrm D \Delta(\lambda)=\begin{bmatrix} 1 & 0 \\0 & 1-\tau\tilde a \textrm e^{-\lambda\tau}\end{bmatrix}.
\end{align}
Hence, the normalization condition can be written with (\ref{eq:pq}) as
\begin{align}
\frac 1 {\alpha_q\alpha_p}\stackrel !=&\begin{bmatrix} \imag c & \omega\end{bmatrix} 
\begin{bmatrix}
1 & 0 \\
0 & 1-\tau \tilde a \textrm e^{-\imag \omega \tau} 
\end{bmatrix}
 \begin{bmatrix} 1 \\ \imag \omega \end{bmatrix} 
.
\end{align}
Again, relation (\ref{eq:iwaexp}) yields a simplification $\imag \omega (\tilde a \textrm e^{-\imag \omega \tau}) =\omega^2-c-\imag \omega a$, such that
\begin{align}
\frac 1 {\alpha_q\alpha_p}\stackrel !=\beta :=&
 \imag c+\imag \omega^2-\omega\tau (\omega^2-c-\imag \omega a ) 
\nonumber
\\
=& -(\omega \tau )(\omega^2-c) + \imag ((\omega \tau) \omega a +c+\omega^2) .
\label{eq:alpha_p_alpha_q_beta}
\end{align}
Since the nonlinearity $h$ in (\ref{eq:stateSpaceGeneral}) and thus also the higher order terms in (\ref{eq:TaylorStateSpaceGeneral}) depend only on $\xI (t)$, i.e. on $X_{11}$ in (\ref{eq:FirstLyap_system}) and (\ref{eq:FirstLyap_systemTaylor}), equation (\ref{eq:LyapCoeffGeneral}) results in
\begin{align}
L=\frac{1}{\omega}\Re e \Bigg(p^\top\!\Bigg[ \!
&
\begin{bmatrix}
0 \\
-h_e'' \phi_1(0) h_{11,1}(0)
\end{bmatrix} 
+ \frac 1 2 
\begin{bmatrix}
0 \\
-h_e'' \overline \phi_1(0) h_{20,1}(0)
\end{bmatrix} 
\nonumber\\
+& \frac 1 2 
\begin{bmatrix}
0 \\
-h_e''' \phi_1(0)^2 \overline \phi_1(0)
\end{bmatrix} 
\!\Bigg]\Bigg).
\label{eq:L_inter}
\end{align}
Let $e_1=[1,0]^\top$, $e_2=[0,1]^\top$, then according to (\ref{eq:phi_h11_h20})
\begin{align}
\phi_1(0) &=q_1 , \quad \quad
 h_{11,1}(0) = e_1^\top \Delta^{-1}\!(0) 
\begin{bmatrix}
0 \\
-h_e'' q_1 \overline q_1
\end{bmatrix} 
,\nonumber\\
 h_{20,1}(0)&=  e_1^\top \Delta^{-1}\!(2\imag \omega) \begin{bmatrix}
0 \\
-h_e'' q_1 q_1
\end{bmatrix} ,
\end{align}
where $q_1 q_1=q_1 \overline q_1=\alpha_q^2$ and
\begin{alignat}{5}
&\Delta(0)=  \begin{bmatrix}0 & -1  \\ c   & a+\tilde a   \end{bmatrix} ,
\quad&&
\Delta(2\imag \omega)=  \begin{bmatrix} 2\imag \omega & -1  \\ c & 2\imag \omega+a+\tilde a (\textrm{e}^{-\imag \omega\tau} )^2 \end{bmatrix}
\nonumber
\\
&e_1^\top\Delta^{-1}\!(0)e_2=  \frac 1 c ,
&&
e_1^\top\Delta^{-1}\!(2\imag \omega)e_2= \frac 1 {\det \Delta(2\imag \omega)}
.&
\end{alignat}
Solving (\ref{eq:iwaexp}) for $\imag \omega \tilde a (\textrm{e}^{-\imag \omega\tau} ) =\omega^2-c-\imag \omega a$ and  (\ref{eq:compRe}) for $(\omega^2-c)^2=\tilde a (1-(a/\tilde a)^2)$, the needed determinant is given by
\begin{align}
\det \Delta(2\imag \omega)=&-4\omega^2+2\imag \omega a +2\imag \omega \tilde a (\textrm{e}^{-\imag \omega\tau} )^2 +c 
\\
=&-4\omega^2+2\imag \omega a - \frac {2 \imag }{ \omega \tilde a}(\omega^2-c-\imag \omega a)^2 +c 
 \\
=&-3\omega^2-(\omega^2-c)-4\frac a{\tilde a} (\omega^2-c)
\nonumber \\
& +2\imag \left(\omega a  \left(1+ \frac a {\tilde a} \right) -\frac{(\omega^2-c)^2}{\omega \tilde a}\right) 
 \\
=&-3\omega^2-\left(1+4\frac a{\tilde a} \right)(\omega^2-c)
\nonumber \\
& +2 \imag \left( \omega\left(a+ \frac {a^2}{\tilde a}\right) - \omega \tilde a \left(1 - \frac{a^2}{\tilde a^2}\right)  \right) 
 \\
=&-3\omega^2-\left(1+4\frac a{\tilde a} \right)(\omega^2-c)
\nonumber \\
& 
+2 \imag \omega \left( a- \tilde a + 2\frac {a^2}{\tilde a}  \right).
\label{eq:detDelta2iw}
\end{align}
With $p_2=\alpha_p\omega$, $\omega \in \mathbb R$ and $q_1=\alpha_q$, (\ref{eq:L_inter}) results in
\begin{alignat}{5}
L=&\Re e \Big(\alpha_p\Big[
&& 
-h_e'' \phi_1(0) h_{11,1}(0)
- \frac 1 2 
h_e'' \overline \phi_1(0) h_{20,1}(0)
\nonumber\\
&&&- \frac 1 2 
h_e''' \phi_1(0)^2 \overline \phi_1(0)
\Big]\Big) 
\nonumber\\
=&\Re e \Bigg(\alpha_p\Bigg[ 
&&h_e'' \alpha_q \frac{1}{c} h_e''\alpha_q \overline \alpha_q
+ \frac 1 2 
h_e'' \overline \alpha_q \frac{1}{\det \Delta(2\imag \omega)}  h_e'' \alpha_q^2
\nonumber\\
&&&- \frac 1 2 
h_e''' \alpha_q ^2 \overline \alpha_q 
\Bigg]\Bigg)  .
\end{alignat}
According to (\ref{eq:alpha_p_alpha_q_beta}), $\alpha_p\alpha_q=\beta^{-1}$, and with 
 $c=h_e'$ 
the first Lyapunov coefficient can be evaluated as
\begin{alignat}{4}
L=&\Re e \Bigg( 
\frac 1 2 \beta^{-1}|\alpha_q|^2  (h''_e)^2
\Bigg[ \frac 2 {h_e'} + \frac 1 {\det \Delta(2\imag \omega)}- \frac{h_e'''}{(h_e'')^2} \Bigg] \Bigg).
\label{eq:resL_general}
\end{alignat}
The decisive terms for the sign of $L$ are stated in (\ref{eq:LyapCoeffDelayedDampingGeneralNonlinearity}).
\end{proof}
\begin{remark}[Noninvariance of $L$]
The value of the first Lyapunov coefficient $L$ in (\ref{eq:resL_general}) depends on $\vert \alpha_q\vert$, i.e.\ on the arbitrary normalization of $q$. However, this ambiguity does not affect the sign of $L$, which alone is decisive in bifurcation analysis (cmp. Kuznetsov\cite{Kuznetsov.1998}, page 99).
\end{remark}
\begin{remark}[Transformed Equation]
There are some mathematical issues discussed in Hale and Verduyn Lunel\cite{Hale.1993} if the delay value is considered as bifurcation parameter. However, by a time scale transformation $t=\hat t \tau$ in combination with $y(\hat t \tau)=\hat y(\hat t)$, it can be achieved that the bifurcation parameter $\tau$ contributes only to coefficients, while the delay value is fixed at one in
\begin{align}
\hat y ''(\hat t)+a  \tau  \hat y'(\hat t)+\tilde a \tau \hat y'( \hat t-1) + \tau^2 \sin(\hat y(\hat t))) &= \tau^2\inh.
\end{align}
Analogously to (\ref{eq:iwaexp}), the characteristic equation at purely imaginary roots $\hat\lambda=i\hat\omega$ is 
\begin{align}
i\hat \omega\tau(a+\tilde a \textrm e ^{-i\omega}) =\omega^2- \tau^2c.
\end{align}
Comparing imaginary parts results in $a+\tilde a \cos(\hat \omega)=0$. Hence, there are two sequences $(\hat\omega_{1n})_{n\in\mathbb N}$, $(\hat\omega_{2n})_{n\in\mathbb N}$
\begin{align}
\hat\omega_{1n}=&\arccos \left(-\frac a {\tilde a}\right)+2\pi n
\\
\hat\omega_{2n}=&-\arccos \left(-\frac a {\tilde a}\right)+2\pi (n+1)
\end{align}
which are related by $\hat\omega_{in}=\omega_i \tau_{in}$, $i\in\{1,2\}$ to the untransformed result (\ref{eq:omega_tau}).
Using $\sin(\omega_{1n,2n})= \pm \sqrt{1-
{a^2}/ {\tilde a^2}}$ in the real parts yields
\begin{align}
\left(\frac {\hat \omega} \tau \right)^2 \mp  \tilde a \sqrt{1-\frac {a^2} {\tilde a^2}}  \left(\frac {\hat \omega} \tau \right) -c&=0  
\end{align}
and indeed leads analogously to (\ref{eq:compRe}) to ${\hat \omega_{in}}/{ \tau_{in}}=\omega_i$ in (\ref{eq:omega12}). Thus, of course, the delay values given in (\ref{eq:tau12n}) are confirmed. 
Furthermore, the inverse normalization factor $\beta$ in (\ref{eq:alpha_p_alpha_q_beta}), derivatives of the nonlinearity $h$, as well as $\det \Delta (2\imag \omega)$ in (\ref{eq:detDelta2iw}) are scaled by $\tau^2$. As a consequence, if formula (\ref{eq:resL_general}) is applied to the transformed equation, all scalings cancel out. Hence, the coefficients of the untransformed equation can directly be used in (\ref{eq:resL_general}) .
\end{remark}


\nocite{*}
\bibliography{Literatur}

\end{document}